\documentclass[12pt]{article}

\usepackage{amsfonts,latexsym,amssymb,amsmath,amsthm,authblk,color}

\date{\today}
\def \al{\alpha}
\def \be{\beta}
\def \ga{\gamma}
\def \dl{\delta}
\def \ep{\varepsilon}

\def \la{\lambda}

\def \ph{\varphi}
\def \om{\omega}
\def \Ga{\Gamma}

\def \Om{\Omega}

\def \operatorname#1{\mathop{\rm #1}}

\def\div{\operatorname{div}}
\def\osc{\operatorname{osc}}
\def\osc2{\operatorname{osc^2}}

\def\cd{\partial}

\def\Q0{Q(x_0,t_0,R)}
\def\0{{x_0,t_0,R}}
\def\build#1_#2{\mathrel{\mathop{\kern 0pt#1}\limits_{#2}}}

\newcommand{\pint}{\mathop{\int\limits{\hspace{-4mm}}-}\limits}
\newtheorem{theorem}{Theorem}[section]

\newtheorem{proposition}{Proposition}[section]
\newtheorem{lemma}{Lemma}[section]
\newtheorem{definition}{Definition}[section]

\begin{document}
\title{Scalar  elliptic   equations   with a singular drift}
\author{Misha Chernobai, Timofey Shilkin}
\date{\today}
\maketitle \abstract{We investigate the weak solvability and
properties of weak solutions to the Dirichlet problem for a scalar
elliptic equation $-\Delta u + b^{(\al)}\cdot \nabla u= f$ in a
bounded domain $\Om\subset \Bbb R^2$ containing the origin, where
$f\in   W^{-1}_q(\Om) $ with  $q>2$ and $b^{(\al)}:=b-\al
\frac{x}{|x|^2}$,   $b$ is a divergence--free vector field and
$\al\in \Bbb R$ is a parameter.}

\maketitle

\bigskip

\section{Introduction and Main Results}\label{Intro}

\bigskip  Assume  $\Om\subset \Bbb R^2 $ is a bounded simply connected domain with a
$C^1$--smooth boundary $\cd \Om$ and $0\in \Om$. We  consider the
following boundary value problem:
\begin{equation}
\left\{ \quad \gathered  -\Delta u + \Big(b -
\al \frac{x}{|x|^2}\Big)\cdot \nabla u \ = \ -\div f \qquad\mbox{in}\quad \Om, \\
u|_{\cd \Om} \ = \ 0. \qquad \qquad
\endgathered\right.
\label{Equation_2D}
\end{equation}
Here  $u:  \Om\to \Bbb R$, $b:\Om\to \Bbb R^2$ and  $f: \Om\to \Bbb
R^2$ are given functions and $\al\in \Bbb R$ is a parameter. We
always assume that the vector field $b$ satisfies the assumptions
\begin{equation}
b\in L_{2,w}(\Om), \qquad \div b = 0 \quad \mbox{in} \quad \mathcal
D'(\Om), \label{Assumptions_b_2D}
\end{equation}
 where $L_{2,w}(\Om)$ is a weak Lebesgue space equipped with the norm
\begin{equation}
\| b\|_{L_{2,w}(\Om)} \ := \ \sup\limits_{\la>0} \la~|\{~x\in \Om:
~|b(x)|>\la~\}|^{\frac 12}. \label{Weak_Lebesgue}
\end{equation}
Note that
$$
\div \frac{x }{|x |^2} \ = \  2\pi \dl_0 \quad \mbox{in} \quad
\mathcal D'(\Om),
$$
where $\dl_0$ is the delta-function concentrated at $x=0$. We also
define the total drift
\begin{equation}
b^{(\al)} \ := \ b-\al \, \frac{x }{|x |^2}. \label{Full_Drift}
\end{equation}
 Then from
\eqref{Assumptions_b_2D} we obtain
$$
b^{(\al)}\in L_{2,w}(\Om), \qquad \| b^{(\al)}\|_{ L_{2,w}(\Om)} \
\le \ C   \left(|\al|+ \| b\|_{ L_{2,w}(\Om)} \right) ,
$$
and the  divergence of $b^{(\al)}$ is sign-definite:  $\div
b^{(\al)} \le 0$ in the sense of distributions if $\al\ge 0$ and
$\div b^{(\al)}>0$ if $\al<0$. The negative sign of $\div b^{(\al)}$
means that the drift term generates the positive increase to the
quadratic form of our differential operator while the negative sign
of $\div b^{(\al)}$ means opposite.
 The goal of this paper is to investigate the existence and
 properties of weak solutions to the problem \eqref{Equation_2D}.

Our research is partly motivated by axially symmetric problems for
the 3D Navier-Stokes equations. In the study of  such problems the
equation
\begin{equation}
\cd_t u -\Delta u + \Big(v -\al \, \frac{x'}{|x'|^2}\Big)\cdot
\nabla u \ = \ 0 \qquad\mbox{in}\quad \Bbb R^3\times (0,T)
\label{NSE}
\end{equation}
plays an important role.  Here $v=v(x,t)$ is the divergence-free
velocity field, $x' = (x_1, x_2,0)^T$, and $u= u(x,t)$ is some
auxiliary scalar function. For example, for axially symmetric
solutions without swirl (i.e. if $v(x,t) = v_r(r,z,t){  e}_r +
v_z(r,z,t){ e}_z$ where ${ e}_r$, ${ e}_\ph$, ${  e}_z$ is the
standard cylindrical basis) the equation \eqref{NSE} is satisfied
for $\al= 2$ and $u= \frac{\om_\ph}r$, where $\om_\ph:=v_{r,z} -
v_{z,r}$ and $r=|x'|$. In the case of general axially symmetric
solutions $v(x,t) = v_r(r,z,t){  e}_r + v_\ph(r,z,t){ e}_\ph +
v_z(r,z,t){  e}_z$ the equation \eqref{NSE} holds for $\al=-2$ and
$u=rv_\ph$.

It is well-known in the Navier-Stokes theory (see, for example,
\cite{KNSS},  \cite{Seregin_Book}, \cite{Seregin_Shilkin_UMN},
\cite{SS}, \cite{Tsai_Book}) that while in the case $\al> 0$ some
results like Liouville-type theorems assume no special conditions on
the solutions $u$ to the equation \eqref{NSE} besides a proper decay
of the drift $v$, the analogues results in the case $\al<0$ require
the additional condition $u|_{\Ga}=0$ where $\Ga:=\{x\in \Bbb R^3:\,
x_1=x_2=0\}$. Our equation \eqref{Equation_2D} can be considered as
a 2D    elliptic model for the general equation \eqref{NSE} and one
of the goals of the present paper is to investigate from the general
point of view the role  the condition $u(0)=0$ (which is modeling
the condition $u|_{\Ga}=0$ in the 3D situation) plays in the theory.

During the last few decades the problem \eqref{Equation_2D} was
intensively studied in the case of the divergence-free drift (i.e.
for $\al=0$), see, for example, the papers \cite{Filonov},
\cite{Filonov_Shilkin}, \cite{IKR}, \cite{Lis_Zhang},  \cite{MV},
\cite{NU}, \cite{SSSZ}, \cite{Vicol}, \cite{Surn}, \cite{Zhang},
\cite{Zhikov} and references there. Papers devoted to the
non-divergence free drifts are not so numerous (see \cite{Kang_Kim},
\cite{Kim_Kim}, \cite{Tsai} for the references). In the general
situation the  drift   can be   decomposed into divergence-free and
potential parts.   So, our problem \eqref{Equation_2D} can also be
interpreted as a model problem containing a singular potential part
of the drift whose divergence is sign-defined.

Our work was motivated to a certain extent by the recent paper
\cite{Tsai} where the problem \eqref{Equation_2D} was studied in the
case of $\Om\subset \Bbb R^n$ with $n\ge 3$ and $b\in L_{n,w}(\Om)$
such that $\div b\in L_{\frac n2,w}(\Om)$, $\div b\le 0$. Our
results concerning the case $\al\ge 0$ can be viewed as a 2D version
of results in \cite{Tsai}. The main technical difference between
\cite{Tsai} and our work is that in the 2D situation the divergence
of the drift is not an integrable function but only a measure.

\bigskip
We define the bilinear form $\mathcal B_\al[u,\eta]$ by
\begin{equation}
\mathcal B_\al[u,\eta] \ := \ \int\limits_\Om  \eta \,
b^{(\al)}\cdot \nabla u  ~dx. \label{Bilinear_Form}
\end{equation}
This form is well-defined at least for $u\in W^1_p(\Om)$ with $p>2$
and $\eta\in L_q(\Om)$ with $q>\frac {2p}{p-2}$. Note that as
$b^{(\al)}$   is singular the drift term in \eqref{Bilinear_Form}
generally speaking is not integrable for $u\in W^1_2(\Om)$. So,
instead of the standard notion of weak solutions from the energy
class $W^1_2(\Om)$ following \cite{Tsai} we introduce the definition
of $p$-weak solutions:

\begin{definition}
Assume $p>2$, $f\in L_p(\Om)$ and $b$ satisfies
\eqref{Assumptions_b_2D}.
 We say $u$ is a $p$-weak solution to the problem
\eqref{Equation_2D} if $u\in \overset{\circ}{W}{^1_p}(\Om)$ and $u$
satisfies the relation
\begin{equation}
\gathered  \int\limits_\Om \nabla u \cdot \nabla\eta ~dx  \ + \
\mathcal B_\al[u,\eta] \ = \ \int\limits_\Om f\cdot \nabla \eta~dx,
\qquad \forall~\eta\in \overset{\circ}{W}{^1_2}(\Om).
\endgathered
\label{Identity_2D}
\end{equation}

\end{definition}
 From the imbedding theorem we obtain that for $p>2$ every  $p$-weak solution to the problem \eqref{Equation_2D}
 is
H\" older continuous:
$$u\in C^{1-\frac 2p}(\bar
\Om).$$ Moreover,  the  form $\mathcal B_\al[u,u]$ in
\eqref{Identity_2D} possesses the following property.

\begin{proposition}\label{Q_Form} Assume $b$ satisfies \eqref{Assumptions_b_2D}.
Then for any $\al\in \Bbb R$ and any $v\in
\overset{\circ}{W}{^1_p}(\Om)$ with $p>2$ the
 quadratic form
$\mathcal B_\al[v,v]$ satisfies the identity
\begin{equation}
\mathcal B_\al[v,v] \ = \   \pi \al\, |v(0)|^2.
\label{Quadratic_Form}
\end{equation}
\end{proposition}

\begin{proof}   Taking into account \eqref{Assumptions_b_2D} and $\div \frac x{|x|^2}=0$ in $\Om\setminus \{ 0\}$  from \eqref{Bilinear_Form} integrating by parts we
obtain
$$
\mathcal B_\al[v,v] \ = \ -\lim\limits_{\ep\to 0}
\int\limits_{\Om\setminus B_\ep}   \al\frac{x}{\, |x|^2}\cdot \nabla
\frac{|v|^2}2 ~dx \ = \ \lim\limits_{\ep\to 0} \frac \al{2\ep}
\int\limits_{\cd B_\ep} |v|^2~ds \ = \ \pi \al \, |v(0)|^2.
$$
Here we denote $B_\ep:=\{~x\in \Bbb R^2: \, |x|<\ep~\}$.
\end{proof}

To demonstrate effects concerning  weak solvability of the problem
\eqref{Equation_2D} we take $f\equiv 0$, $b\equiv 0$,  and consider
the equation
\begin{equation}
\left\{ \quad \gathered -\Delta u - \al \, \frac{x}{|x|^2}\cdot
\nabla u \ = \ 0  \qquad\mbox{in} \quad B,
\\
u|_{\cd B} \ = \ 0,  \qquad \qquad \endgathered \right.
\label{Radial}
\end{equation}
 in the
unit ball $B:=\{~x\in \Bbb R^2:\, |x|<1~\}$.  For
 radial solutions $u(x) = v(|x|)$  the equation \eqref{Radial}
 reduces to the ODE for the
Darboux-type operator
$$
v''(r) \ + \ \frac{\al+1}r ~v'(r) \ = \ 0, \qquad r\in (0,1).
$$
The family of solutions to this ODE is $v(r)= c_1 ~r^{-\al}+ c_2$ if
$\al\not=0$ and $c_1\ln r + c_2$ if $\al=0$. If $\al\ge 0$ then the
only weak radial solution $u(x)=v(|x|)$ to the problem
\eqref{Radial} such that $u\in W^1_2(B)$ is identically zero. In
contrast, if $\al<0$ then we have a one-parameter family $v(x)=
c(r^{|\al|}-1)$ of non-trivial $p$-weak solutions to the problem
\eqref{Radial} with some $p=p(\al)$, $p>2$. In particular, for
$\al<0$ $p$-weak solutions of the problem \eqref{Radial} are
non-unique, see also \cite{Mosc} for a related result. Nevertheless,
if for $\al<0$  we add the additional condition $u(0)=0$  to the
problem \eqref{Radial}
 then  any radial  $p$-weak
solution  to  \eqref{Radial} satisfying this requirement is
identically zero. This hints that in the case of $\al<0$ the
condition $u(0)=0$ can serve as an additional requirement that
provides uniqueness for the  problem \eqref{Equation_2D}. Our goal
is to investigate existence of $p$-weak solutions satisfying this
requirement.

\bigskip
The main results of the present paper are the following two
theorems:

\begin{theorem}\label{Theorem_1} Assume $\al \ge 0$,  $b$ satisfies
\eqref{Assumptions_b_2D} and   $q>2$. Then there exists $p>2$
depending only on $q$, $\Om$, $\al$ and $\| b\|_{L_{2,w}(\Om)}$ such
that  for any $f\in L_q(\Om)$  there exists a unique $p$-weak
solution $u$ to
  the problem \eqref{Equation_2D}. Moreover, this solution satisfies the estimate
\begin{equation}
\| u\|_{W^1_p(\Om)} \ \le \ C~\| f\|_{L_q(\Om)},
\label{Main_Estimate}
\end{equation}
 with the
constant $C$ depending only on $q$, $p$, $\al$, $\Om$ and $\|
b\|_{L_{2,w}(\Om)}$.
\end{theorem}

\begin{theorem}\label{Theorem_2} Assume $\al < 0$,  $b$ satisfies
\eqref{Assumptions_b_2D} and   $q>2$. Then there exists $p>2$
depending only on $q$, $\Om$, $\al$ and $\| b\|_{L_{2,w}(\Om)}$ such
that  for any $f\in L_q(\Om)$  there exists a unique $p$-weak
solution $u$ to
  the problem \eqref{Equation_2D} satisfying the condition  $u(0)=0$.
  Moreover, this solution satisfies the estimate
  \eqref{Main_Estimate}
with the constant $C$ depending only on $q$, $p$, $\al$, $\Om$ and
$\| b\|_{L_{2,w}(\Om)}$.

\end{theorem}

Note that the assumption $\cd \Om$ is of class $C^1$  in Theorems
\ref{Theorem_1} and \ref{Theorem_2} actually is made only for
brevity and our proofs can be extended for more general domains.

\medskip\noindent
{\bf Acknowledgement.} The research of the second author was
supported by RFBR, grant  20-01-00397.

\medskip
Our paper is organized as follows: in Section \ref{A priori
estimates Section} we prove some a priori estimates of $p$-weak
solutions including higher integrability results, in Section
\ref{Holder_Continuity Section} we prove the estimate of the H\"
older norm of $p$-weak solutions in the case of $\al<0$ which turns
out to be a crucial step in our proof of the higher integrability of
$\nabla u$ if $\al<0$, in Section \ref{Proof of Main Results} we
present the proofs of Theorems \ref{Theorem_1} and \ref{Theorem_2},
in Appendix we recall some known facts and prove an approximation
result for divergence-free functions from $L_{2,w}(\Om)$.

\medskip

In the paper we use the following notation. For any $a$, $b\in
\mathbb  R^n$ we denote by $a\cdot b$ their  scalar product in
$\mathbb R^n$. We denote by $L_p(\Omega)$ and $W^k_p(\Omega)$ the
usual Lebesgue and Sobolev spaces. $C_0^\infty(\Om)$ is the space of
smooth functions compactly supported in $\Om$. The space
$\overset{\circ}{W}{^1_p}(\Omega)$ is the closure of
$C_0^\infty(\Omega)$ in $W^1_p(\Omega)$ norm.  The space of
distributions on $\Om$ is denoted by $\mathcal D'(\Om)$. By $C(\bar
\Omega)$ and $C^\alpha(\bar \Omega)$, $\alpha\in (0,1)$ we denote
the spaces of continuous and H\" older continuous functions on $\bar
\Omega$.   For $q\in [1, +\infty)$ let
$L_{q,w}(\Om)$ be a weak Lebesgue space.

 The symbols $\rightharpoonup$ and $\to $ stand for the
weak and strong convergence respectively. We denote by $B_R(x_0)$
the ball in $\mathbb R^n$ of radius $R$ centered at $x_0$ and write
$B_R$ if $x_0=0$. We write also $B$ instead of $B_1$. For
$\om\subset \Bbb R^2$ we  denote the average of $f$ over $\om$ by
$$
(f)_\om \ := \ -\!\!\!\!\!\!\int\limits_\om f~dx \ = \ \frac
1{|\om|}~\int\limits_\om f~dx.
$$
For $f$, $g: \Bbb R^2\to \Bbb R$ we denote  by $f*g$  the
convolution
$$
(f*g)(x) \ := \ \int\limits_{\Bbb R^2} f(x-y)g(y)~dy.
$$

\newpage
\section{Higher Integrability Results} \label{A priori estimates
Section}
\setcounter{equation}{0}

\bigskip
In this section we recall a priori estimates for weak solutions to
the problem \eqref{Equation_2D}. Our first theorem is the energy
estimate:

\begin{lemma} \label{Enegy_norm_estimate}
Assume  $\al\in \Bbb R$ and  $b$  satisfies
\eqref{Assumptions_b_2D}. Let $u$ be a  $p$-weak solution to the
problem  \eqref{Equation_2D},  corresponding to some  $f\in
L_p(\Om)$ with $p>2$. In the case of $\al<0$ assume additionally
that $ u(0) = 0$. Then the estimate
 \begin{equation}
\| u\|_{W^1_2(\Om)} \ \le \ C~\| f\|_{L_2(\Om)} \label{Energy}
 \end{equation}
holds with some constant  $C>0$ depending only on $\Om$.
\end{lemma}

\begin{proof}
We take $\eta=u$ in \eqref{Identity_2D} and use
\eqref{Quadratic_Form}.
\end{proof}

The next result is the global estimate of $L_\infty$-norms of
$p$-weak solutions.

\begin{lemma} \label{L_infinity_estimate}
Assume $\al\in \Bbb R$ and  $b$ satisfies \eqref{Assumptions_b_2D}.
Let $u\in \overset{\circ}{W}{^1_p}(\Om)$ be  a $p$-weak solution to
the problem \eqref{Equation_2D} corresponding to some $f\in
L_q(\Om)$ with $q>2$. In the case of $\al<0$ assume additionally
that $ u(0)  =  0$.  Then $u\in L_\infty(\Om)$ and the estimate
 \begin{equation}
\| u\|_{L_\infty(\Om)} \ \le \ C~\| f\|_{L_q(\Om)}
\label{A_priori_estimate_negative_alpha}
 \end{equation}
 holds  with some constant $C$ depending only on $\Om$ and $q$.

\end{lemma}

\begin{proof} Assume $k\ge 0$ and take $\eta=(u-k)_+:=\max\{0, u-k\}$ in
\eqref{Identity_2D}. Note that if   $u(0)=0$ and  $k\ge 0$ then
$(u-k)_+(0)=0$. From \eqref{Quadratic_Form} we obtain
$$
\mathcal B_\al[u, (u-k)_+] \ = \ \mathcal B_\al[(u-k)_+, (u-k)_+] \
= \ \pi \al \, |(u-k)_+|^2(0).
$$
The last term in non-negative for $\al\ge 0$ and vanishes if
$\al<0$, $u(0)=0$, $k\ge 0$. Hence from \eqref{Identity_2D} we
obtain
$$
\int\limits_{A_k} |\nabla u|^2~dx  \ \le \ \int\limits_{A_k} f \cdot
\nabla u~dx, \qquad \forall~k\ge0.
$$
where we denote $A_k \ := \ \{ ~x\in \Omega: ~u(x)>k~\}$. The rest
of the proof goes as in the usual elliptic theory. Applying the  H\"
older inequality we obtain
$$
\int\limits_{A_k} |\nabla u|^2~dx \ \le \
\|f\|_{L_q(\Omega)}^2~|A_k|^{ \dl} , \qquad \forall~k\ge
 0,
$$
where  $ \dl:=2\left(\frac 12-\frac 1q\right)>0$. This inequality
yields the following estimate, see \cite[Chapter II, Lemma 5.3]{LU},
$$
\operatorname{esssup}\limits_\Omega  u _+ \ \le \ C(q,
\Omega)~\|f\|_{L_q(\Omega)} .
$$
where $u_+:=\max\{ 0,u\}$. The estimate of
$\operatorname{esssup}\limits_\Omega u_-$ where $u_-:=\max\{ 0,
-u\}$  can be obtained in a similar way if we replace $u$ by $-u$.
\end{proof}

Now we prove  a  higher integrability
result for $\nabla u$ in the case of $\al\ge 0$. Our proof is based
on  the reverse H\" older inequality for $\nabla u$, see, for
example, \cite[Chapter V]{Giaquinta}.

\begin{theorem} \label{Higher_Integrability_1}
Assume $\al\ge 0$,  $b$ satisfies \eqref{Assumptions_b_2D} and
$q>2$. Then there exists $p>2$ depending only on $q$, $\Om$, $\al$
and $\| b\|_{L_{2,w}(\Om)}$ such that for any $f\in L_q(\Om)$ and
any $p$-weak solution $u$  to the problem \eqref{Equation_2D} the
estimate \eqref{Main_Estimate}
 holds  with some constant $C$ depending only on $\Om$, $q$, $\al$ and $\|
 b\|_{L_{2,w}(\Om)}$.
\end{theorem}

\begin{proof}
 First we consider  an internal point $x_0\in \Om$. Take
  $R>0$ such that $B_{2R}(x_0)\subset \Om$ and choose a cut-off
function $\zeta\in C_0^\infty(B_{2R}(x_0))$ so  that $\zeta\equiv 1$
on $B_R(x_0)$ and $\|\nabla\zeta\|_{L_\infty(\Bbb R^2)}\le C/R$.
Denote $\bar u = u-(u)_{B_{2R}(x_0)}$. Taking $\eta= \zeta^2\bar u$
 in \eqref{Identity_2D} and using the relation
$$
\mathcal B_\al[u,\zeta^2\bar u] \ = \ \mathcal B_\al[\zeta \bar
u,\zeta \bar u] \ - \ \int\limits_{B_{2R}(x_0)}b^{(\al)} \cdot \zeta
|\bar u |^2 \nabla \zeta~dx
$$
 after routine computations  we arrive at
\begin{equation}
\gathered \int\limits_{B_R(x_0)}|\nabla u|^2~dx \ + \  \mathcal
B_\al[\zeta \bar u,\zeta \bar u] \ \le \ \frac
14\int\limits_{B_{2R}(x_0)}|\nabla u|^2~dx \ - \
\int\limits_{B_{2R}(x_0)}b^{(\al)} \cdot \zeta |\bar u |^2 \nabla
\zeta~dx \ +  \\  + \ C~ \int\limits_{B_{2R}(x_0)}(|f|^2+|\bar u
|^2)~dx.
\endgathered
\label{Reverse_Holder_1}
\end{equation}
Taking into account Proposition \ref{Q_Form} and the identity
$\zeta(0)=1$ we obtain
\begin{equation}
 \mathcal B_\al[\zeta \bar
u,\zeta \bar u]  \ = \  \pi  \al\, | \bar u (0)|^2
\label{B_identity}
\end{equation}
As $\al\ge 0$ we have
\begin{equation}
 \mathcal B_\al[\zeta \bar
u,\zeta \bar u]  \ \ge \ 0 \label{Positive_1}
\end{equation}
and hence we can drop this term from \eqref{Reverse_Holder_1}.
 For the second term on the right-hand side of \eqref{Reverse_Holder_1} we
 get
$$
\gathered  \int\limits_{B_{2R}(x_0)}b^{(\al)} \cdot \zeta |\bar u
|^2 \nabla \zeta~dx \ \le \ \frac
CR~\|b^{(\al)}\|_{L^{2,\infty}}(\Om) \| |\bar u |^2
\|_{L^{2,1}(B_{2R}(x_0))} \  \le \\  \le \ \frac
CR~\|b^{(\al)}\|_{L^{2,\infty}(\Om)} \| \bar u
\|_{L^{4,2}(B_{2R}(x_0))}^2.
\endgathered
$$
Using the H\" older inequality for Lorentz spaces  (see
\cite[Proposition 2.1]{Kozono_Yamazaki}, also cited in \cite{ONeil},
\cite[Section 4.1]{Garfakos}.) we obtain:
$$
\| \bar u \|_{L^{4,2}(B_{2R}(x_0))} \ \le \ C~\| 1\|_{L^{20,
10/3}(B_{2R}(x_0))}~ \| \bar u \|_{L^{5,5}(B_{2R}(x_0))}.
$$
Taking into account the property of the indicator function $\|
\chi_E \|_{L^{p,s}(\Om)} = (p/s)^{1/s}|E\cap \Om|^{1/p}$ we
derive
$$
 \| \bar u
\|_{L^{4,2}(B_{2R}(x_0))} \ \le \ C~|B_{2R}|^{\frac 1{20}}~ \| \bar
u \|_{L_5(B_{2R}(x_0))}.
$$
Hence we arrive at
$$
\int\limits_{B_{2R}(x_0)}b^{(\al)} \cdot \zeta |\bar u |^2 \nabla
\zeta~dx \ \le \ \frac {C}{R^{4/5}}~\|b^{(\al)}\|_{L_{2,w}(\Om)} \|
\bar u \|_{L_5(B_{2R}(x_0))}^2.
$$
Using the imbedding theorem we
estimate
\begin{equation}
\gathered  \int\limits_{B_{2R}(x_0)}b^{(\al)} \cdot \zeta |\bar u
|^2 \nabla \zeta~dx \ \le \ \frac {C
}{R^{4/5}}~\|b^{(\al)}\|_{L_{2,w}(\Om)} \| \nabla u \|_{L_{\frac
{10}7}(B_{2R}(x_0))}^2.
\endgathered
\label{Drift_Term_1}
\end{equation}
Using \eqref{Positive_1} and \eqref{Drift_Term_1} from
\eqref{Reverse_Holder_1} we obtain
\begin{equation}
\gathered \left(  \ \  \  -  \!\!\!\!\!\!\!\!\!\!
\int\limits_{B_R(x_0)}|\nabla u|^2~dx\right)^{\frac 12} \ \le \
\frac 12\,\left(  \ \  \ -   \!\!\!\!\!\!\!\!\!\!\!
\int\limits_{B_{2R}(x_0)}|\nabla u|^2~dx\right)^{\frac  12} + C\,
\left(  \ \  \  -  \!\!\!\!\!\!\!\!\!\!\!
\int\limits_{B_{2R}(x_0)}|f|^2~dx\right)^{\frac  12} \ + \\ +  \ C\,
\Big(1+\|b^{(\al)}\|_{L_{2,w}(\Om)}^{1/2}\Big) \left(  \ \  \ -
 \!\!\!\!\!\!\!\!\!\!\! \int\limits_{B_{2R}(x_0)}|\nabla
u|^{\frac{10}7}~dx\right)^{\frac  7{10}}.
\endgathered
\label{Reverse_1}
\end{equation}

\medskip
 Now we derive similar estimates near the boundary.
Extend $u$ and $f$  by zero outside $\Om$ and denote this extension
by $\tilde u$ and $\tilde f$. Note that as $\cd \Om$ is
$C^1$--smooth there exists $d>0$ depending on $\Om$ such that for
any $R<d$ and any $x\in \cd \Om$ \ $|B_R(x) \setminus \Om |> \frac
14 |B_R|$. Without loss of generality one can assume
$d<\operatorname{dist}\{ 0,\cd \Om\}$. Take $x_0\in \cd \Om$ and
$R<d/2$. Choose a cut-off function $\zeta\in
C_0^\infty(B_{2R}(x_0))$ so that $\zeta\equiv 1$ on $B_R(x_0)$. Then
$\zeta^2 u\in \overset{\circ}{W}{^1_2}(\Om_{2R}(x_0))$ where we
denote $\Om_R(x_0):= B_R(x_0)\cap \Om$. Taking $\eta= \zeta^2 u$
 in \eqref{Identity_2D} after
routine computations  we arrive at
\begin{equation}
\gathered \int\limits_{B_R(x_0)}|\nabla \tilde u|^2~dx \ + \
\mathcal B_\al[  u,\zeta^2  u] \ \le \ \frac
14\int\limits_{B_{2R}(x_0)}|\nabla \tilde u|^2~dx  \ + \ C~
\int\limits_{B_{2R}(x_0)}(|\tilde f|^2+|  \tilde u |^2)~dx.
\endgathered
\label{Reverse_Holder_2}
\end{equation}
As $0\not \in \Om_{2R}(x_0)$ we conclude $\div b^{(\al)}=0$ in
$\mathcal D'(\Om_{2R}(x_0))$ and hence
$$
\mathcal B_\al[u,\zeta^2  u] \ = \  - \
\int\limits_{B_{2R}(x_0)}b^{(\al)} \cdot \zeta | \tilde u |^2 \nabla
\zeta~dx.
$$
Repeating the arguments in the internal case we obtain
$$
\int\limits_{B_{2R}(x_0)}b^{(\al)} \cdot \zeta | \tilde u |^2 \nabla
\zeta~dx \ \le \ \frac {C}{R^{4/5}}~\|b^{(\al)}\|_{L_{2,w}(\Om)} \|
\tilde u \|_{L_5(B_{2R}(x_0))}^2.
$$
As $|\{ \, x\in B_{2R}(x_0): \, \tilde u(x)=0 \, \}| \ge \frac
14|B_{2R}|$ we can apply Lemma \ref{Poincare_inequality} and obtain
$$
\| \tilde u \|_{L_5(B_{2R}(x_0))} \ \le \ C~ \| \nabla \tilde u
\|_{L_{\frac {10}7}(B_{2R}(x_0))} \ = \ C~\| \nabla u \|_{L_{\frac
{10}7}(\Om_{2R}(x_0))}
$$
with some absolute constant $C>0$. Hence from
\eqref{Reverse_Holder_2} we obtain
\begin{equation}
\gathered \left(  \ \  \  -  \!\!\!\!\!\!\!\!\!\!
\int\limits_{\Om_R(x_0)}|\nabla u|^2~dx\right)^{\frac 12} \ \le \
\frac 12\,\left(  \ \  \ -   \!\!\!\!\!\!\!\!\!\!\!
\int\limits_{\Om_{2R}(x_0)}|\nabla u|^2~dx\right)^{\frac  12} + C\,
\left(  \ \  \  -  \!\!\!\!\!\!\!\!\!\!\!
\int\limits_{\Om_{2R}(x_0)}|f|^2~dx\right)^{\frac  12} \ + \\ +  \
C\, \Big(1+\|b^{(\al)}\|_{L_{2,w}(\Om)}^{1/2}\Big) \left(  \ \  \ -
 \!\!\!\!\!\!\!\!\!\!\! \int\limits_{\Om_{2R}(x_0)}|\nabla
u|^{\frac{10}7}~dx\right)^{\frac  7{10}}
\endgathered
\label{Reverse_2}
\end{equation}
which holds for any $x_0\in \cd \Om$ and any $R<d$. Combining the
interior and boundary estimates in
the standard way we obtain \eqref{Reverse_2} for any $x_0\in \bar
\Om$ and any $R<d/2$. This estimate is the reverse H\" older
inequality for $\nabla u$, see \cite[Chapter V]{Giaquinta}. Hence
there exists $p\in (2,q]$ such that   $\nabla u\in L_p(\Om)$ and the
following estimate holds:
$$
\| \nabla u\|_{L_p(\Om)} \ \le \ C~(\| \nabla u\|_{L_2(\Om)} + \|
f\|_{L_q(\Om)})
$$
with some constant $C>0$ depending only on $q$, $\Om$, $\al$ and $\|
b\|_{L_{2,w}(\Om)}$. Combining this estimate with \eqref{Energy} we
obtain \eqref{Main_Estimate}.
\end{proof}

The estimate of the H\" older norm of solutions to the problem
\eqref{Equation_2D} is the crucial step in our proof of the higher
integrability of $\nabla u$ for $\al<0$. This estimate was
originally obtained in \cite{CSTY_1}, \cite{CSTY_2}, see also
\cite{NU}. As the mentioned papers formally deal with a bit stronger
assumption  on the divergence-free part of the drift $b$ (which in
2D-case corresponds to $b\in L_2(\Om)$ instead of $b\in
L_{2,w}(\Om)$ in our case) we outline the proof of this result.

\begin{theorem} \label{Holder_Continuity_2}
Assume $\al< 0$,  $b$ satisfies \eqref{Assumptions_b_2D} and $q>2$.
Then there exists $\mu\in (0,1)$ depending only on $q$, $\Om$, $\al$
and $\| b\|_{L_{2,w}(\Om)}$ such that for any $f\in L_q(\Om)$, any
$p>2$  and any $p$-weak solution $u$  to the problem
\eqref{Equation_2D} the estimate
\begin{equation}
\| u\|_{C^\mu(\bar \Om)} \ \le \ C~\big(\|f\|_{L_q(\Om)} +\|
u\|_{L_\infty(\Om)}\big), \label{Holder_estimate}
\end{equation}
holds with some constant $C$ depending only on $\al$, $\Om$, $q$ and
$\|
 b\|_{L_{2,w}(\Om)}$.
\end{theorem}

\begin{proof}
For any $\al\in \Bbb R$ the drift $b^{(\al)}$ is divergence free in
$\Om\setminus \{0\}$:
$$
\div b^{(\al)} =0 \quad \mbox{in} \quad
\mathcal D'(\Om\setminus \{0\}).
$$
It is well-known that solutions of elliptic equations with the
divergence free drifts $b^{(\al)}\in L_{2,w}(\Om\setminus B_R)$  are
H\" older continuous, see, for example, \cite{NU}, \cite{Fri_Vicol},
\cite{IKR}. Hence for any $R>0$ such that $B_{2R}\Subset \Om$ we
have $u\in C^\mu(\overline{\Om\setminus B_R})$ and for any $x_0\in
\bar \Om \setminus B_R$ and any $0<\rho < R$ the following estimate
holds:
\begin{equation}
\operatorname{osc}\limits_{\Om\cap B_\rho(x_0)} u \ \le \
C\Big(\frac \rho R\Big)^\mu \left(\operatorname{osc}\limits_{\Om
\cap B_R(x_0)} u \ + \ \| f\|_{L_q(\Om \cap B_{R}(x_0))}\right)
\label{Together_1}
\end{equation}
with some constant $\mu\in (0,1)$ and $C>0$ depending only on $q$,
$\Om$, $\al$ and $\|b\|_{L_{2,w}(\Om)}$.

On the other hand, for any $0<\rho <R$ we have the estimate
\begin{equation}
\operatorname{osc}\limits_{  B_\rho} u \ \le \ C\Big(\frac \rho
R\Big)^\mu \left(\operatorname{osc}\limits_{ B_R } u \ + \ \|
f\|_{L_q(  B_{R} )}\right). \label{Together_2}
\end{equation}
Though this estimate is not new (see, for example \cite[Theorem
10.7]{Tsai_Book}), for the reader's convenience we present its proof
 in the next section, see Theorem
\ref{Oscillation_Estimate} below. Now take arbitrary $x_0\in \Om$
and denote $r:=|x_0|$. For $r<R$ we obtain $B_r(x_0)\subset
B_{2r}\subset B_{2R}$ and
$$
\gathered  \operatorname{osc}_{B_\rho(x_0)} u \ \le \ C~\Big(\frac
\rho r\Big)^\mu \, \Big(\operatorname{osc}_{B_r(x_0)} u \ + \
\|f\|_{L_q(B_r(x_0))}\Big) \ \le \\ \le \ C~\Big(\frac \rho
r\Big)^\mu \, \Big( \operatorname{osc}_{B_{2r}} u \ + \
\|f\|_{L_q(B_{2r})}\Big) \ \le \ C~\Big(\frac \rho r\Big)^\mu \,
\Big(\frac { r}{ R}\Big)^\mu \Big( \operatorname{osc}_{B_{2R}} u \ +
\ \|f\|_{L_q(B_{2R})}\Big)\ \le
\\ \le \ C~\Big(\frac {\rho}R\Big)^\mu ~\Big(\|u\|_{L_\infty(\Om)} \ + \
\|f\|_{L_q(\Om)}\Big).
\endgathered
$$
The last estimate together with \eqref{Together_1},
\eqref{Together_2} implies \eqref{Holder_estimate}.
\end{proof}

The last theorem of the present section is the analogue of Theorem
\ref{Higher_Integrability_1} in the case of $\al<0$.

\begin{theorem} \label{Higher_Integrability_2}
Assume $\al< 0$,  $b$ satisfies \eqref{Assumptions_b_2D} and $q>2$.
Then there exists $p>2$ depending only on $q$, $\Om$, $\al$ and $\|
b\|_{L_{2,w}(\Om)}$ such that for any $f\in L_q(\Om)$ and any
$p$-weak solution $u$  to the problem \eqref{Equation_2D} satisfying
the additional assumption $u(0)=0$
  the estimate \eqref{Main_Estimate}
 holds  with some constant $C>0$ depending only on $\Om$, $q$, $\al$ and $\|
 b\|_{L_{2,w}(\Om)}$. Moreover, without loss of generality one can assume
 the  constant  $C$ in \eqref{Main_Estimate} is a
 non-decreasing function with respect to $\| b\|_{L_{2,w}(\Om)}$.
\end{theorem}

\begin{proof} Assume $\al<0$ and $u$ is a $p$-weak solution to the problem
\eqref{Equation_2D}. From Theorem \ref{Holder_Continuity_2} we
obtain $u\in C^\mu(\bar \Om)$ with some $\mu\in (0,1)$ and hence
integrating by parts and taking into account $u(0)=0$ we obtain
\begin{equation}
\int\limits_\Om \frac{x}{|x|^2} \cdot \nabla u \, \eta~dx \ = \ -
\int\limits_\Om u  \, \frac{x}{|x|^2} \cdot \nabla  \eta~dx, \qquad
\forall~\eta\in C_0^\infty(\Om). \label{RHS_divergent}
\end{equation}
The last relation means that $u$ is a $p$-weak solution to the
problem
\begin{equation}
\left\{ \quad \gathered  -\Delta u + b \cdot \nabla u \ = \ -\div
g \qquad\mbox{in}\quad \Om, \\
u|_{\cd \Om} \ = \ 0, \qquad u(0)=0, \qquad \qquad
\endgathered\right.
\label{Equation_2D_modified}
\end{equation}
where
$$
g \ = \ f - \al \, \frac{x}{|x|^2} \, u.
$$
Note that $g\in L_{q_0}(\Om)$ for any $q_0\in \left(2, \min \left\{
q, \frac{2}{1-\mu}\right\}\right)$ and
\begin{equation}
\| g\|_{L_{q_0}(\Om)} \ \le \ C   \left(\| f\|_{L_q(\Om)} + \,
|\al|\, \| u\|_{C^\mu(\bar \Om)}\right) \label{g_estimate}
\end{equation}
with a constant $C>0$ depending only on $\mu$ and $\Om$. Applying
Theorem \ref{Higher_Integrability_1} to the problem
\eqref{Equation_2D_modified} we conclude there exist constants $p>2$
and $C>0$ depending only on $q_0$, $\Om$ and $\|b \|_{L_{2,w}(\Om)}$
such that
$$
\| u\|_{W^1_p(\Om)} \ \le \ C\, \| g\|_{L_{q_0}(\Om)}.
$$
Taking into account \eqref{g_estimate} and \eqref{Holder_estimate}
we obtain the estimate \eqref{Main_Estimate}.
\end{proof}

\newpage
\section{H\"older Continuity of $p$--Weak Solutions} \label{Holder_Continuity
Section}
\setcounter{equation}{0}

\bigskip
\medskip
The main result of this section is the following estimate:

\medskip

\begin{proposition}\label{Oscillation_Estimate}
Assume $\al< 0$,  $b$ satisfies \eqref{Assumptions_b_2D},
$b^{(\al)}$ is defined in \eqref{Full_Drift}. Assume  $q>2$, $f\in
L_q(\Om)$ and $B_{2R}\Subset \Om$. Assume $p>2$ and $u$ is a
$p$-weak solution to
\begin{equation}
-\Delta u + b^{(\al)} \cdot \nabla u \ = \ -\div f \qquad\mbox{in}
\quad \Om \label{Equation}
\end{equation}
such that $u(0)=0$. Then there exist $\mu\in (0,1)$ and $C>0$
depending only on $q$, $\al$, $\|b\|_{L_{2,w}(\Om)}$ such that for
any $0<\rho <R$ the following estimate holds:
\begin{equation}
\operatorname{osc}\limits_{  B_\rho} u \ \le \ C\Big(\frac \rho
R\Big)^\mu \left(\operatorname{osc}\limits_{ B_R } u \ + \ \|
f\|_{L_q(  B_{R} )}\right). \label{Oscillation_Estimate_Main}
\end{equation}
\end{proposition}

 \medskip

Our proof is based on  the ideas of \cite[Section 10.3]{Tsai_Book},
see also \cite{NU} where a similar result is obtained. We split the
proof of Theorem \ref{Oscillation_Estimate} onto several steps. We
start from the Caccioppolli-type inequality:

\medskip
\begin{lemma}\label{Caccioppolli} Assume $\al< 0$,  $b$ satisfies
\eqref{Assumptions_b_2D} in $\Om=B_2$ and $p>2$. Then for any $f\in
L_p(B_2)$, any  $p$-weak solution $u$ to the equation
\eqref{Equation} in $B_2$ such that $u(0)\le 0$, and any $\zeta\in
C_0^\infty(B_2)$ the following estimate holds:
\begin{equation}
\gathered  \int\limits_{B_2} \zeta^2|\nabla u_+|^2dx \ \le \ C
~\int\limits_{ B_2 } |   u_+|^2 \Big(|\nabla \zeta|^2 +
|b^{(\al)}|\, |\nabla \zeta|\Big) \, dx \ + \   C ~\int\limits_{ B_2
} \zeta^2 |   f|^2 \, dx.
\endgathered \label{Caccioppolli_Inequality}
\end{equation}
Here $u_+:=\max\{ u, 0\}$ and $C>0$ is some absolute constant.
\end{lemma}

\begin{proof} We take $\eta:=\zeta^2 u_+$ in \eqref{Identity_2D}.
The convective term satisfies  the identity
$$
\mathcal B_\al[u,\zeta^2 u_+] \ = \ \mathcal B_\al[ \zeta u_+ ,
\zeta u_+] \ - \ \int\limits_{B_2} \zeta \, |u_+|^2 \, b^{(\al)}
\cdot \nabla \zeta \, dx.
$$
As $u(0)\le 0$ we obtain $u_+(0)=0$ and hence from
\eqref{Quadratic_Form} we conclude
$$
\mathcal B_\al[ \zeta u_+ , \zeta u_+] \ = \ \pi \al \,
\zeta^2(0)|u_+(0)|^2 \ = \ 0.
$$
Now the result follows by the H\" older and Young inequalities.
\end{proof}

Now we proceed with the following maximum estimate:

\medskip
\begin{lemma}\label{Maximum estimate} Assume $\al<0$, $b$ satisfies \eqref{Assumptions_b_2D} in $\Om=B$,  $q>2$ and $f\in L_q(B)$.
 Then there is a constant $C>0$ depending only on $q$, $\al$ and $\|
b\|_{L_{2,w}(B)}$ such that for any $p>2$ and any  $p$-weak solution
$u$ of the equation \eqref{Equation}  in $B$ satisfying $u(0)\le 0$
the following estimate holds:
\begin{equation}
\sup\limits_{B_{1/2}}  u_+ \ \le \ C~\Big(
 -\!\!\!\!\!\!\int\limits_{B} |  u_+|^{q_1}~dx\Big)^{1/q_1} \, + \
C \, \| f\|_{L_q(B)}, \label{Maximum estimate_1}
\end{equation}
where $q_1:=\max\left\{ 5, \frac{2q}{q-2}\right\}$.
\end{lemma}

\begin{proof}  The proof of \eqref{Maximum estimate_1} follows by application of the standard Moser iteration technique.
Let $\beta \ge 0$ be arbitrary and denote $$\bar u := u_++k, \qquad
z := \bar u^{\frac{\be+2}2}, \qquad w \ := \ {\bar u}^{\frac
{\beta+2} 2}-k^{\frac{\be+2}2},$$ where we take $k:= \|
f\|_{L_q(B)}$ if $f\not\equiv 0$ and $k=0$ otherwise. Assume
$\zeta\in C_0^\infty(B)$ is an arbitrary cut-off function and take $
\eta  =  \zeta^2 \bar u^{\frac \be 2}w$ in \eqref{Identity_2D}. We
have
$$
\int\limits_B \nabla u \cdot \nabla\eta~dx \ \ge \ \frac 2{\be +2}
~\int\limits_B \zeta^2 |\nabla w|^2 dx \   - \ \frac 4{\be +2} ~
\int\limits_B |w \nabla\zeta  \cdot \zeta\nabla w| \, dx,
$$
$$
\mathcal B_\al [u,\eta] \ = \ \frac 2{\be +2}~ \mathcal B_\al[w,
\zeta^2 w],
$$
$$
\int\limits_B f\cdot \nabla \eta~dx \ \le \ 2~\int\limits_B \zeta \,
|f| \, \bar u^{\frac \be2} \,  \Big( \zeta \, |\nabla w|  +  |\nabla
\zeta|\,   w \Big)~dx.
$$
As $u(0)\le 0$ we obtain $w(0)=0$. Combining all above inequalities
from \eqref{Identity_2D} similar to \eqref{Caccioppolli_Inequality}
 we obtain
$$
\gathered  \int\limits_{B } |\zeta \nabla   w |^2dx \ \le \ C
~\int\limits_{ B  }     w^2 \Big(|\nabla \zeta|^2 + |b^{(\al)}|\,
|\nabla \zeta|\Big) \, dx \ + \   C \, (\be +2)^2 ~\int\limits_{ B }
\zeta^2 |  f|^2 \bar u^\be \, dx .
\endgathered
$$
Note that $k\le \bar u$, $w\le z$ and $\nabla z=\nabla w$. So, in
the case   $f\not\equiv 0$ from the last inequality we obtain
$$
\gathered  \int\limits_{B } |\zeta \nabla   z|^2dx \ \le \ C
~\int\limits_{ B  }     z^2 \Big(|\nabla \zeta|^2 + |b^{(\al)}|\,
|\nabla \zeta|\Big) \, dx \ + \   C\, (\be +2)^2 ~\int\limits_{ B  }
\zeta^2 \Big| \frac fk \Big|^2  z^2\, dx .
\endgathered
$$
Recall  that $\big\| \frac
fk\big\|_{L_q(B)}=1$. Hence the last term can be estimated by the
H\" older inequality:
$$
\int\limits_{ B  } \zeta^2 \Big| \frac fk \Big|^2  z^2\, dx  \ \le \
\Big\| \frac fk \Big\|^2_{L_q(B)} \|\zeta z\|_{L_{\frac{2q}{q-2}}(B)
}^2 \   \le \ C~ \|\zeta   z\|_{L_{q_1}(B) }^2,
$$
and   we arrive at
$$
\gathered  \int\limits_{B } |\nabla  (\zeta  z)|^2dx \ \le \ C \,
\int\limits_{ B  }     z^2 \Big(|\nabla \zeta|^2 + |b^{(\al)}|\,
|\nabla \zeta| \Big) \, dx \ + \ C\, (\be +2)^2 \| \zeta
z\|_{L_{q_1}(B) }^2.
\endgathered
$$
 Applying the
imbedding theorem
$$
\| \zeta z\|_{L_{2q_1}(B )} \ \le \ C \,  \| \nabla (\zeta
z)\|_{L_2(B )},
$$
taking  arbitrary $\frac 12\le r<R\le 1$, and choosing $\zeta\in
C_0^\infty(B_R)$ so that $\zeta\equiv 1$ on $B_r$, $|\nabla
\zeta|\le C/(R-r)$,   we get
$$
\gathered \int\limits_{ B  }     z^2 |b^{(\al)}|\, |\nabla \zeta| \,
dx \ \le  \ \frac{C}{R-r} \,
\|b^{(\al)}\|_{L_{2,w}(B_R)}\|z\|_{L^{4,2}(B_R)}^2
\endgathered
$$
 and
$$
\gathered \int\limits_{ B  }     z^2 |\nabla \zeta|^2~dx \ \le \ \|
|\nabla\zeta|^2\|_{L_{2,w}(B_R\setminus B_r)} \|
z^2\|_{L^{2,1}(B_R)} \ \le \ \frac{C}{(R-r)^{3/2} }\,
\|z\|_{L^{4,2}(B_R)}^2.
\endgathered
$$
Combining these estimates we obtain
$$
\gathered  \| z\|_{L_{2q_1}(B_r)} \ \le \ \frac{C}{\sqrt{R-r}}
\left(   \|b^{(\al)}\|_{L_{2,w}(B)}^{\frac 12} +(R-r)^{-\frac1 4}
\right)~ \| z\|_{L^{4,2}(B_R)} \ + \  C\, (\be +2) \, \|
z\|_{L_{q_1}(B_R) }.
\endgathered
$$
Using   the H\" older inequality for Lorentz norms  and taking into
account  $R\in (\frac 12, 1]$ we conclude
$$
  \| z\|_{L^{4,2}(B_R)} \ \le \ C\, R^{2(1/4-1/{q_1})} \, \|
z\|_{L_{q_1}(B_R)} \ \le \ C \, \| z\|_{L_{q_1}(B_R)}.
$$
Applying  the Cauchy inequality  we derive the estimate
$$
\| z\|_{L_{2q_1}(B_r)} \ \le \ C \, \left( \be +2  +  K(R-r)^{-1}
\right)~ \| z\|_{L_{q_1}(B_R)}, \qquad K:=
\|b^{(\al)}\|_{L_{2,w}(B)}^{\frac 12}+1,
$$
which is valid for any $\frac 12\le r<R\le 1$.
 Hence we have
\begin{equation}
\|  \bar u\|_{L_{2 \gamma q_1 }(B_r)} \ \le \ C^{\frac{1}\gamma}
\left(2\ga   +  K(R-r)^{-1} \right)^{\frac 1{ \gamma}}~ \| \bar
u\|_{L_{\gamma q_1}(B_R)} \label{Iterate}
\end{equation}
with an arbitrary $\gamma\ge 1$, $\gamma:=\frac{\beta+2}2$. Denote
$s_0=q_1$, $s_m:=2 s_{m-1} $,  and denote also $R_m= \frac 12
+\frac1{2^{m+1}}$. Taking $r=R_m$, $R=R_{m-1}$,
$\gamma=\frac{s_{m-1}}{q_1}$   in \eqref{Iterate} we obtain
$$
\gathered  \| \bar u\|_{L_{s_m }(B_{R_m})}  \ \le \  \exp (c_0 {m
}{2^{-m } } )  ( 1  + K
  )^{ 2^{1-m} }~ \|
\bar u\|_{L_{s_{m-1}}(B_{R_{m-1}})}.
\endgathered
$$
Iterating this inequality we get
$$
\gathered  \| \bar u\|_{L_{s_m }(B_{1/2})}  \ \le \  C_0 \, ( 1  + K
)^2 \, \| \bar u\|_{L_{q_1}(B)}.
\endgathered
$$
with some positive constant $C_0$ independent on $m$. Taking $m\to
\infty$ we arrive at  \eqref{Maximum estimate_1}. 
\end{proof}

\medskip
\begin{lemma} \label{Thin_set_lemma}
Assume $\al<0$, $b$ satisfies \eqref{Assumptions_b_2D} in $\Om=B$,
$q>2$ and $f\in L_q(B)$.
 Then there are constants $C>0$ and $\ep_0\in (0,1)$ depending only on $q$, $\al$ and $\|
b\|_{L_{2,w}(B)}$ such that for any $p>2$ and any  $p$-weak solution
$u$ of the equation \eqref{Equation}  in $B$ such that $u(0)\le 0$
if
$$
|\{ \, x\in B: \, u(x)>  0 \, \}| \ \le \  \ep_0\, |B|
$$
then
$$
\sup\limits_{B_{1/2}} u_+ \ \le \ \frac 12~ \sup\limits_{B } u_+ \ +
\  C\, \| f\|_{L_q(B )}.
$$
\end{lemma}

\medskip
\begin{proof} From Lemma \ref{Maximum estimate} we obtain
\eqref{Maximum estimate_1} with $q_1:=\max\left\{ 5,
\frac{2q}{q-2}\right\}$. Then we have
$$
 \gathered  C~\Big( -\!\!\!\!\!\!\int\limits_{B }
|u_+|^{q_1}dx\Big)^{1/q_1} \  \le  \ C~\sup\limits_{B } u_+
~\frac{|\{ \, x\in B : u(x)> 0 \, \}|^{1/q_1}}{|B |^{1/q_1}} \ \le \
C\, \ep_0^{1/q_1} \, \sup\limits_{B } u_+.
\endgathered
$$
Choosing $C\ep_0^{1/q_1} =\frac 12$ we obtain the required
statement. \
\end{proof}

Next we prove that a small value subset has density less then 1.

\begin{lemma}     \label{Small_value_subset_has_density_less_then_1}
Assume $\al<0$, $b$ satisfies \eqref{Assumptions_b_2D} in $\Om=B$.
 Then for any $p>2$, any $g\in L_p(B )$ and any $p$-weak solution
$v$ to the problem
\begin{equation} -\Delta v + b^{(\al)} \cdot
\nabla v \ = \ -\div g \qquad\mbox{in} \quad B  \label{Equation_v}
\end{equation}
satisfying the assumptions
\begin{equation}
0\le v\le 2 \quad \mbox{in} \quad B , \qquad v(0)\ge 1,
\label{Assumptions_v}
\end{equation}
if
\begin{equation}
\|g\|_{L_2(B)} \ \le \    c_\star \, |\al| \label{Assumption_g},
\qquad \mbox{where} \qquad c_\star:=\frac{\sqrt{\pi}}4,
\end{equation}
then  the following estimate holds:
\begin{equation}
 |\{ \, x\in B: \, v(x)  >\la  \, \}| \ > \ \ga.
\label{Nonzero_density}
\end{equation}
Here $\la$, $\ga \in (0,1)$ are some  constants depending on $\al$
and $\| b\|_{L_{3/2}(B)}$ in the explicit way described in the proof
below.
\end{lemma}

\begin{proof}  Assume there exist $g$ and $v$
satisfying \eqref{Equation_v}, \eqref{Assumptions_v},
\eqref{Assumption_g} such that
\begin{equation}
   |\{ \, x\in B : \, v(x)> \la \,
\}| \ \le \ \ga.
 \label{le_lambda}
\end{equation}
Integrating by parts in \eqref{Identity_2D} for any $\eta\in
C_0^\infty(B)$  we obtain
\begin{equation}
2\pi |\al| v(0)\eta(0) \ = \ - \int\limits_{B } v\Big(\Delta \eta +
b^{(\al)}\cdot \nabla \eta  \Big)~dx \ - \ \int\limits_{B } g\cdot
\nabla \eta~dx. \label{v(0)_estimate}
\end{equation}
Note that $v(0) \ge 1$. Choose  $\eta $  so that $\eta = 1$ on
$B_{1/2}$,  $ \|\nabla\eta\|_{L_2(B)}\le 4\sqrt{\pi}$ and $\|
\eta\|_{C^2(\bar B)}\le c_{\star\star}$ where $c_{\star\star}>0$ is
some sufficiently large absolute constant. Then from
\eqref{Assumption_g} and the H\" older inequality we obtain
$$
\Big| ~\int\limits_{B } g\cdot \nabla \eta ~dx~\Big|  \ \le \ \pi
|\al|.
$$
Hence from \eqref{v(0)_estimate} we obtain
\begin{equation}
\pi |\al|   \ \le \  \Big| \, \int\limits_{B } v\Big(\Delta \eta  +
b^{(\al)}\cdot \nabla \eta   \Big)~dx ~\Big|. \label{together}
\end{equation}
Denote
$$
B [v> \la] := \{~x\in B : \, v(x)> \la~\}, \qquad B [v\le \la] :=
\{~x\in B : \, v(x)\le \la~\}.
$$
From the H\" older inequality we obtain
$$
 \int\limits_{B [v> \la]} v\Big(\Delta \eta  +
b^{(\al)}\cdot \nabla \eta   \Big)~dx \  \le   \   \|
v\|_{L_\infty(B )} \, \| \eta \|_{C^2(\bar B )} \,
\Big(1+\|b^{(\al)}\|_{L_{\frac 32}(B )} \Big) |B [v>\la]|^{\frac
13}.
$$
Taking into account \eqref{Assumptions_v}, \eqref{le_lambda}  and
$\|\eta\|_{C^2(B)}\le c_{\star\star}$  we conclude
$$
\Big| \int\limits_{B [v> \la]} v\Big(\Delta \eta  + b^{(\al)}\cdot
\nabla \eta  \Big)~dx \, \Big|  \  \le   \  2c_{\star\star} \,
\Big(1+\|b^{(\al)}\|_{L_{\frac 32}(B )} \Big) \ga^{\frac 13}.
$$
  On the other hand
$$
\gathered  \Big|\int\limits_{B [v\le \la]} v\Big(\Delta \eta  +
b^{(\al)}\cdot \nabla \eta  \Big)~dx \Big| \  \le \ c_{\star\star}
\, \Big(1+\|b^{(\al)}\|_{L_{1}(B )} \Big) \, \la.
\endgathered
$$
Finally, we obtain
$$
|\al|   \  \le   \  2c_{\star\star} \,
\Big(1+\|b^{(\al)}\|_{L_{\frac 32}(B )} \Big)  \ga^{\frac 13} \ + \
c_{\star\star} \, \Big(1+\|b^{(\al)}\|_{L_{1}(B )} \Big) \, \la.
$$
This inequality leads to  the contradiction if we fix values of
$\la$, $\ga \in (0,1) $ so that
$$
 2c_{\star\star} \,
\Big(1+\|b^{(\al)}\|_{L_{\frac 32}(B )} \Big) \ga^{\frac 13} \ + \
 c_{\star\star} \, \Big(1+\|b^{(\al)}\|_{L_1(B )}  \Big) \, \la  \ < \ \pi
|\al|.
$$

\end{proof}

\begin{lemma} \label{Level_increasing_lemma}
Assume $\al<0$, $b$ satisfies \eqref{Assumptions_b_2D} in $\Om=B_2$
and let $\la$, $\ga \in (0,1)$ be the constants  defined in Lemma
\ref{Small_value_subset_has_density_less_then_1}. Then for any
$\ep>0$ there exists $s_0\in \Bbb N$ depending only on $\ep$, $\la$,
$\al$, $\|b \|_{L_{2,w}(B)}$ such that for any $p>2$, any $g\in
L_p(B )$ and any $p$-weak solution $v$ to the problem
\eqref{Equation_v} in $B_2$ satisfying \eqref{Assumptions_v} in
$B_2$ either
\begin{equation}
  |\{ \, x\in B: \, v(x) \, \le \, 2^{-s_0} \la \, \}| \ \le \
\ep \, |B| \label{First_case}
\end{equation}
or
\begin{equation}
  2^{- s_0} \la \ \le \ \frac{1}{c_\star |\al|} \, \|
g\|_{L_2(B_2)}.\label{Second_case}
\end{equation}
\end{lemma}

\begin{proof}
Assume $s_0\in \Bbb N$ is arbitrary and   $g$ and  $v$ satisfy
\eqref{Equation_v}, \eqref{Assumptions_v}. If
$$
\|g\|_{L_2(B_2)} \ \ge \    c_\star  |\al|
$$
then \eqref{Second_case}   holds as $\la\in (0,1)$. Assume now
\eqref{Assumption_g} is valid. Then \eqref{Nonzero_density} is true.

Let $k\in (0,1)$ be arbitrary. Denote $u:=k-v$. Then $u$ satisfies
\eqref{Equation} and $u(0)=k-v(0)\le 0$. Then for any $\zeta\in
C_0^\infty(B_2)$  from \eqref{Caccioppolli_Inequality} we obtain
$$
 \int\limits_{B_2}  \zeta^2|\nabla(v-k)_-|^2\,
dx \ \le \ C  \int\limits_{B_2}  (v-k)_- ^2\Big( |\nabla \zeta|^2 +
|b^{(\al)}|~|\nabla\zeta| \Big)dx + C \, \int\limits_{B_2} |g|^2\,
dx.
$$
Choosing $\zeta$ so that $\zeta\equiv 1$ in $B$, $|\nabla\zeta|\le
4$,  for any $k\in (0,1)$ we obtain
\begin{equation}
  \int\limits_{B }  |\nabla(v-k)_-|^2\,
dx  \ \le \ C\,  \int\limits_{B_2}  (v-k)_- ^2\Big( 1 +  |b^{(\al)}|
\Big)dx + C \, \int\limits_{B_2} |g|^2\, dx.
\label{Caccioppolli_Inequality_1}
\end{equation}
 For $ s=0,1,2, \ldots$ we denote $k_s:= 2^{-s}\la$ and
$$
 A_s:=\{ \, x\in B : \, v(x)< k_s \, \} = \{ \, x\in B : \, (v-k_s)_-(x)>0 \,
 \}.
$$
Applying the De Giorgi inequality
(see \cite[Chapter II \S 3, Lemma 3.9]{LU}) we obtain
$$
(k_{s}-k_{s+1}) \, | A_{s+1} |^{1/2} \ \le \  \frac{ C }{| B
\setminus  A_s    |} \ \int\limits_{A_s\setminus A_{s+1}} |\nabla
v(x)|~dx.
$$
As $B \setminus A_s\supset B \setminus A_0$ and $|B \setminus
A_0|\ge  \ga|B |$ using the H\" older inequality  we obtain
$$
 2^{-s-1} \la \, |A_{s+1} |^{1/2} \ \le \ \frac{C}{ \ga} \, |A_s\setminus
A_{s+1}|^{1/2} \, \Big( \int\limits_{A_s\setminus A_{s+1}} |\nabla
v|^2~dx\Big)^{1/2}.
$$
Hence
$$
 |A_{s+1} | \ \le \ \frac{C }{\la^2\ga^2 } \, 2^{2s+2} \, |A_s\setminus
A_{s+1}|~\int\limits_{B} |\nabla (v-k_s)_-|^2~dx.
$$
Using \eqref{Caccioppolli_Inequality_1} with $k=k_s$ we arrive at
$$
 |A_{s+1} | \ \le \ C\, \frac{ 2^{2s} }{ \la^2\ga^2  } \, \, |A_s\setminus
A_{s+1}| \int\limits_{B_2} \left[ (v-k_s)_- ^2\big( 1 + |b^{(\al)}|
\big) +    |g|^2\right]\, dx.
$$
As $v\ge 0$ in $B_2$ we have  $(v-k_s)_-\le k_s=2^{-s}\la $ in $B_2$
and
$$
 |A_{s+1} |  \ \le \ \frac{C}{ \ga^2   }  \, |A_s\setminus
A_{s+1}| \int\limits_{B_2} \left(    1 +   |b^{(\al)}|   +
\frac{2^{2s}  }{\la^2  } |g|^2\right)\, dx.
$$
Taking the sum from $s=0$ to $s=s_0-1$ and using inclusions  $
B\supset A_0\supset A_1\supset \ldots \supset A_{s_0}$ we obtain
\begin{equation}
 s_0 |A_{s_0}|  \ \le \ \frac{C_0}{ \ga^2 } \, \Big( 1 +
 \|b^{(\al)}\|_{L_1(B_2)} + \frac{2^{2s_0}}{\la^2  } \|
g\|^2_{L_2(B_2)}\Big)
 \label{A_0}
 \end{equation}
with some absolute positive constant $C_0$.
 Now let us fix the value $s_0\in \Bbb N$ so that
\begin{equation}
\frac{C_0}{\ga^2 s_0 } \Big( 1+ c_\star^2 |\al|^2  +
\|b^{(\al)}\|_{L_1(B_2)} \Big) \ \le \ \ep \, |B|
 \label{Choice_s_0}
 \end{equation}
and consider two cases. If
$$
\frac{2^{ s_0}}{\la   } \, \| g\|_{L_2(B_2)}  \ \ge \ c_\star|\al|
$$
then  \eqref{Second_case} follows.
 Otherwise
$$
\frac{2^{ s_0}}{\la  } \,  \| g\|_{L_2(B_2)}  \ \le \ c_\star |\al|
$$
and from \eqref{A_0}, \eqref{Choice_s_0} we obtain
\eqref{First_case}.
\end{proof}

\medskip
\begin{lemma}\label{Lower_bound}
Assume $\al<0$, $b$ satisfies \eqref{Assumptions_b_2D} in $\Om=B_2$,
$q>2$ and $g\in L_q(B_2)$.  Let $\la$, $\ga\in (0,1)$ be the
constants defined in Lemma
\ref{Small_value_subset_has_density_less_then_1},  $\ep_0\in (0,1)$
be defined in Lemma \ref{Thin_set_lemma} and $s_0\in \Bbb N$ be the
number from Lemma \ref{Level_increasing_lemma} defined in
\eqref{Choice_s_0} for $\ep=\ep_0$. Denote $\dl:=2^{-s_0}\la$,
$\dl\in (0,1)$. Then for any $p>2$ and any $p$-weak solution $v$ to
the problem \eqref{Equation_v} in $B_2$ satisfying
\eqref{Assumptions_v} in $B_2$ either
\begin{equation}
c_\star \, |\al|  \, \dl \ \le \  \| g\|_{L_2(B_2)}
\label{Second_case_2}
\end{equation}
or
\begin{equation}
   \inf\limits_{B_{1/2}} v  \ +   \ c \, \| g\|_{L_q(B)}\
\ge \ \frac \dl2 . \label{First_case_2}
\end{equation}
Here $c>0$ is some   constant depending only on $q$, $\al$ and $\|
b\|_{L_{2,w}(B_2)}$.
\end{lemma}

\begin{proof} If $g$ satisfies \eqref{Second_case} then
\eqref{Second_case_2} follows. Assume now
$$
\| g\|_{L_2(B )}  \ < \ c_\star  |\al|  \, \dl.
$$
Then applying   Lemma \ref{Level_increasing_lemma} we obtain
$$
 |\{ \, x\in B: \, v(x)\le \dl \, \}| \ \le \
\ep_0|B|.
$$
Denote $u= \dl-v$. Then $u(0)= \dl-v(0)<0$ and $u$ satisfies all
assumptions of  Lemma \ref{Thin_set_lemma} with $f=-g$.  Hence we
obtain
$$
\sup\limits_{B_{1/2} } (\dl-v)_+ \ \le \ \frac 12~\sup\limits_{B }
(\dl-v)_+ \ + \ c\, \| g\|_{L_q(B)}\ \le  \ \frac \dl2 \ + \ c \, \|
g\|_{L_q(B)}
$$
which gives
$$
\dl \ - \ \inf\limits_{B_{1/2}} v \ \le \ \frac \dl2 \ + \ c\, \|
g\|_{L_q(B )}.
$$
Lemma \ref{Lower_bound} is proved. \end{proof}

\begin{lemma}\label{Oscillation_Theorem}
Assume $\al<0$, $b$ satisfies \eqref{Assumptions_b_2D} in $\Om=B_2$,
$q>2$ and $f\in L_q(B)$.
 There exist  constants
$\dl\in (0,1)$ and $c>0$ depending only on $q$, $\al$ and $\|
b\|_{L_{2,w}(B_2)}$ such that    for any
 $p>2$ and any  a $p$-weak solution $u$ to
the equation \eqref{Equation} in $B_{2}$ satisfying $u(0)=0$ the
following inequality holds:
\begin{equation}
\operatorname{osc}_{B_{1/2}} u \ \le  \ \Big( \, 1-\frac \dl2 \,
\Big)~ \operatorname{osc}_{B_2 } u  \ + \ c\,  \| f\|_{L_q(B_2 )}.
\label{Osc}
\end{equation}
\end{lemma}

\medskip

\begin{proof} For any $r\in (0,2)$ we denote
$$
\gathered  M( r) := \sup\limits_{B_{ r}} u, \qquad m( r) :=
\inf\limits_{B_{ r}} u, \qquad \om(r) := M(r)-m(r),
\endgathered
$$
Note that \ $m(2)\le 0\le M(2)$. \ Denote $k_0:=\frac {m(2)+M(2)}2$,
and define   $v$ and $g$ so that
$$
\gathered  v(x)  := 2~\frac{u(x)-m(2)}{\om(2)}, \qquad g(x) :=
\ \ \frac{2f(x)}{\om(2)}, \qquad \mbox{if} \qquad  k_0< 0, \\
v(x):= 2~\frac{M(2) - u(x)}{\om(2)}, \qquad
g(x):=-\frac{2f(x)}{\om(2)}, \qquad \mbox{if} \qquad  k_0\ge 0.
\endgathered
$$
Then $v$ is a $p$-weak solution to the equation \eqref{Equation_v}
in $B_2$ satisfying the additional conditions \eqref{Assumptions_v}
in $B_2$. From Lemma \ref{Lower_bound}  we conclude
$$
\gathered \mbox{either} \quad \dl \ \le \ c\, \| g\|_{L_2(B_2)}
\qquad \mbox{or} \quad  \inf\limits_{B_{1/2}} v +   c\, \|
g\|_{L_q(B)}\ \ge  \ \frac \dl2.
\endgathered
$$
In the first case from the H\" older inequality we obtain
$$
\om(2) \ \le \ \frac{ c}{\dl} \, \| f\|_{L_2(B_2)} \ \le \ c(\dl, q)
\, \| f\|_{L_q(B_2)}
$$
and hence \eqref{Osc} follows. In the second case if $k_0<0$ we have
$$
m(1/2) \ - \ m(2)  \ + \    2c\, \| f\|_{L_q(B)} \ \ge  \ \frac \dl2
\, \om(2)
$$
and hence
$$
\Big(1-\frac \dl2\Big) \, \om(2) \ +  \ 2c \, \| f\|_{L_q(B)} \ \ge
\ \om(1/2)
$$
and we obtain \eqref{Osc}. If $k_0\ge 0$ then
$$
M(2) \ - \ M(1/2) \ + \   2c\, \| f\|_{L_q(B)}\ \ge  \ \frac \dl2 \,
\om(2)
$$
and we again obtain \eqref{Osc}.
\end{proof}

\begin{lemma}\label{Oscillation_Theorem_Rescaled}
Assume $\al< 0$,  $b$ satisfies \eqref{Assumptions_b_2D} in $\Om$.
Assume $q>2$, $f\in L_q(\Om)$ and $B_{2R}\Subset \Om$.
 Then there exist  constants
$\dl\in (0,1)$ and $c>0$ depending only on $q$, $\al$ and $\|
b\|_{L_{2,w}(B_{2R})}$ such that for any
 $p>2$
and any $p$-weak solution $u$ to the equation \eqref{Equation} in
$\Om$  satisfying $u(0)=0$   the following inequality holds:
\begin{equation}
\operatorname{osc}_{B_{R/2}} u \ \le  \ \Big( \, 1-\frac \dl2 \,
\Big)~ \operatorname{osc}_{B_{2R}} u  \ + \ c\, R^{1-\frac 2q} \, \|
f\|_{L_q(B_{2R})}. \label{Osc_R}
\end{equation}
\end{lemma}

\medskip
\begin{proof}
For $x\in B_2$ we denote
$$
u^R(x) = u(Rx), \qquad b^{(\al)}_R(x) = Rb^{(\al)}(Rx), \qquad
f^R(x)=  Rf(Rx).
$$
Then $u^R$ is a solution to
$$
-\Delta u^R + b^{(\al)}_R \cdot \nabla u^R \ = \ -\div f^R
\qquad\mbox{in} \quad B_2
$$
and, moreover,
$$
\| b^{(\al)}_R\|_{L_{2,w}(B_2)} \ = \ \|
b^{(\al)}\|_{L_{2,w}(B_{2R})}, \qquad \| f^R\|_{L_q(B_2)} \ = \
R^{1-\frac 2q} \, \| f\|_{L_q(B_{2R})}.
$$
From Lemma \ref{Oscillation_Theorem} we obtain
$$
\operatorname{osc}_{B_{1/2}} u^R \ \le  \ \Big( \, 1-\frac \dl2 \,
\Big)~ \operatorname{osc}_{B_2 } u^R  \ + \ c\,  \| f^R\|_{L_q(B_2
)}
$$
which implies \eqref{Osc_R}.
\end{proof}

Now  the inequality \eqref{Oscillation_Estimate_Main}  follows  from
\eqref{Osc_R} by the standard iteration technique. Proposition
\ref{Oscillation_Estimate} is proved.

\newpage
\section{Proof of Main Results}\label{Proof of Main Results}
\setcounter{equation}{0}

\bigskip
We start with the proof of Theorem \ref{Theorem_1}. To construct a
$p$-weak solution to the problem \eqref{Equation_2D} we approximate
our drift $b^{(\al)}$ by smooth functions using Lemma
\ref{Approximation_result}.  In the case  $\al\ge 0$ our
construction shows that  $p$-weak solutions are always approximative
solutions in the sense they can be obtained as limits of smooth
solutions of the equations with smooth drifts. Note that this is not
true  if $\al<0$ as the approximative solution must satisfy the
maximum principle (locally), but the radial solutions constructed in
the case of $\al<0$ in Section \ref{Intro} do not possess this
property.

\begin{proposition}\label{Positive_T1}
Assume $\al\ge 0$, $b \in C^\infty(\bar \Om)$ satisfies $\div b=0$
in $\Om$  and $q>2$. Assume $f\in L_q(\Om)$ and $\ep>0$. Define
$$
b^{(\al)}_\ep \ := \ b  - \al\, \frac{ x}{|x|^2+\ep^2}.
$$
Then    there exists a unique $q$--weak solution $u  $ to the
problem
\begin{equation}
\left\{ \quad \gathered  -\Delta u  + b^{(\al)}_\ep\cdot \nabla u  \ = \ -\div f  \qquad\mbox{in}\quad \Om, \\
u |_{\cd \Om} \ = \ 0. \qquad \qquad
\endgathered\right.
\label{Equation_2D_approximate}
\end{equation}
Moreover, there exists  $p\in (2,q]$ depending only on $q$, $\Om$,
$\al$ and $\| b\|_{L_{2, w} (\Om)}$ such that  the following
estimate holds:
\begin{equation}
\| u \|_{W^1_p(\Om)} \ \le \ C\, \| f\|_{L_q(\Om)}.
\label{Higher_Integrability_approx}
\end{equation}
Here  $C>0$ is a  constant which can be chosen to be  a
non-decreasing function of $q$, $\Om$, $\al$ and $\| b\|_{L_{2, w}
(\Om)}$.
\end{proposition}

\begin{proof}
The existence and uniqueness of the weak  solution $u \in
W^1_2(\Om)$ for  equation \eqref{Equation_2D_approximate} with the
smooth drift $b^{(\al)}_\ep$ is well-known, see, for example,
\cite{LU}. Moreover, from $L_q$--theory for equations with smooth
coefficients we obtain $u^\ep\in W^1_q(\Om)$.  Note that as $\al\ge
0$ we have $\div b^{(\al)}_\ep \le 0$ in $\Om$. Hence the estimate
\eqref{Higher_Integrability_approx} follows by the arguments similar
to the  proof of Theorem \ref{Higher_Integrability_1}.
\end{proof}

Now we state an   approximation result in weak Lebesgue space:

\medskip
\begin{proposition}\label{Approximation_result}
Let $\Om \subset \Bbb R^2$ be a bounded simply connected domain of
class $C^1$. Assume $b\in L_{2,w}(\Om)$, $\div b=0$ in $\mathcal
D'(\Om)$. Then there exist $b^\ep\in C^\infty(\bar \Om)$, $\div
b^\ep=0$ in $\mathcal D'(\Om)$, such that
\begin{itemize}
\item[{\rm 1)}] $\| b^\ep\|_{L_{2,w}(\Om)} \, \le \, C \, \|
b\|_{L_{2,w}(\Om)}$,
\item[{\rm 2)}] $b^\ep\to b$ a.e. in $\Om$ as $\ep\to 0$.
\end{itemize}
Here the constant $C>0$ depends only on $\Om$.

\end{proposition}

\begin{proof}
Assume $p\in (1,+\infty)$ and denote
$$
J_p(\Om) \ := \ \{ ~u\in L_p(\Om; \Bbb R^2):~\div u = 0 \ \mbox{ in
} \
 \mathcal D'(\Om)~\}.
$$
Consider any bounded simply connected domain $\Om_0\subset \Bbb R^2$
such that $\Om\Subset \Om_0$. Then there exists an extension
operator $T: J_p(\Om)\to J_p(\Bbb R^2)$ such that for all $b \in
J_p(\Om)$ the function $\tilde b:=Tb$ has the following properties
$$
\gathered  \tilde b \equiv 0 \quad \mbox{in} \quad \Bbb R^2\setminus
\Om_0, \qquad \tilde b|_\Om=b, \\ \div \tilde b = 0
\quad\mbox{in}\quad \mathcal D'(\Bbb R^2), \qquad \| \tilde
b\|_{L_p(\Bbb R^2)}\le c_p \|  b\|_{L_p(\Om)}.
\endgathered
$$
The existence of the extension operator $T$ is shown, for example,
in \cite[\S3.1, Theorem 3.1]{Gi_Ra}, and   $L_p$-estimates for this
operator follow from the corresponding estimates for weak solutions
to the Neumann problem involved into the construction of the
operator $T$. Indeed, denote by $E: W^1_{p'}(\Om_0\setminus \Om)\to
W^1_{p'}(\Om_0)$ any bounded linear extension operator, i.e.
$(E\eta)\big|_{\Om_0\setminus \Om}=\eta$. Then the formula
$$
l(\eta) \ :=  \ -\int\limits_\Om b\cdot \nabla \tilde\eta~dx, \qquad
\tilde \eta:=E\eta, \qquad \eta\in W^1_{p'}(\Om_0\setminus \Om)
$$
determines a bounded linear functional on the  Banach space
$W^1_{p'}(\Om_0\setminus \Om)$. Moreover, the condition $\div b =0$
in $\mathcal D'(\Om)$ implies that the functional $l$ is independent
of the extension operator $E$. Hence there exists a unique solution
$\ph\in W^1_p(\Om_0\setminus \Om)$ of the Neumann problem
$$
\gathered
 \int\limits_{\Om_0\setminus \Om}
\nabla \ph\cdot \nabla \eta~dx \ = \ l(\eta), \qquad \forall~\eta\in
W^1_{p'}(\Om_0\setminus \Om), \\
\int\limits_{\Om_0\setminus \Om} \ph~dx = 0 , \qquad \|\nabla
\ph\|_{L_p(\Om_0\setminus \Om)} \ \le \ c~\| b\|_{L_p(\Om)},
\endgathered
$$
and one can take $\tilde b:=b$ in $\Om$ and $\tilde b:= \nabla \ph$
in $\Om_0\setminus \Om$ and then  extend $\tilde b$ by zero to
 the whole $\Bbb R^2$. It is
easy to see that $\tilde b\in J_p(\Bbb R^2)$ and $Tb:=\tilde b$
satisfies all necessary properties.

From the Marcinkiewicz interpolation theorem (see \cite[Theorem
1.4.19]{Garfakos}) we conclude that $T$ is also bounded as an
operator from $L_{2,w}(\Om)$ into $L_{2,w}(\Bbb R^2)$:
$$
\| \tilde b\|_{L_{2,w}(\Bbb R^2)}\le c_{2,w} \|  b\|_{L_{2,w}(\Om)},
\qquad\forall~b\in L_{2,w}(\Om): \qquad \div b=0\quad\mbox{in}\quad
\mathcal D'(\Om).
$$
Let $\om_\ep$ be the standard Sobolev kernel
$\om_\ep(x)=\ep^{-2}\om(x/\ep)$, $\om \in C_0^\infty(B)$,
$\int_{\Bbb R^2}\om(x)~dx=1$, and denote
$$\tilde b^\ep:=\om_\ep* \tilde b, \qquad b^\ep:=\tilde
b^\ep|_\Om.$$ Then for $b\in L_{2,w}(\Om)$ we have
$$
\tilde b^\ep\in C_0^\infty(\Bbb R^2), \qquad  \div \tilde b^\ep = 0
\quad\mbox{in}\quad \mathcal D'(\Bbb R^2), \qquad   \tilde b^\ep\to
  \tilde b \quad \mbox{in}\quad L_p(\Bbb R^2), \qquad \forall~p\in
  [1,2).
$$
Moreover, as the convolution operator $T_\ep \tilde b:=  \om_\ep*
\tilde b \equiv \tilde b^\ep$ is bounded in $L_p(\Bbb R^2)$ for any
$p\in (1,+\infty)$ with $\| T_\ep\|_{L_{p}\to L_{p}}=1$, i.e.
$$
\|  \tilde b^\ep\|_{L_p(\Bbb R^2)}  \ \le  \ \|\tilde b\|_{L_p(\Bbb
R^2)}, \qquad \forall~p\in (1,+\infty),
$$
 from the Marcinkiewicz interpolation theorem  (see \cite[Theorem 1.4.19]{Garfakos}) we conclude that $T_\ep$ is also
bounded as an operator from $L_{2,w}(\Bbb R^2)$ into $L_{2,w}(\Bbb
R^2)$ and the norm $\| T_\ep\|_{L_{2,w}\to L_{2,w}}$ is independent
of $\ep$:
$$
\exists ~M>0: \qquad \| \tilde b^\ep\|_{L_{2,w}(\Bbb R^2)} \ \le \
M~ \| \tilde b\|_{L_{2,w}(\Bbb R^2)}.
 $$
Hence $b^\ep$ possesses all required properties.
\end{proof}

Now we can prove  Theorem \ref{Theorem_1}:

\begin{proof}  Let $p>2$ be the exponent  defined in Proposition \ref{Positive_T1}.
Note that $p$ depends only on $q$, $\Om$, $\al$ and $\| b\|_{L_{2,
w} (\Om)} $. From Proposition \ref{Approximation_result}  we obtain
the existence of a sequence $b_\ep\in C^\infty(\bar \Om)$ such that
\begin{equation}
\div b_\ep = 0 \quad \mbox{in} \quad \Om , \qquad  b_\ep \to b \quad
\mbox{a.e. in} \quad \Om, \qquad \| b_\ep\|_{L_{2,w}(\Om)} \, \le \,
C\, \| b\|_{L_{2,w}(\Om)}. \label{Properties_approx}
\end{equation}
Denote
\begin{equation}
b^{(\al)}_\ep \ := \ b_{\ep}  - \al\, \frac{ x}{|x|^2+\ep^2}
\label{b_alpha_def}
\end{equation}
and note that
\begin{equation}
\| b^{(\al)}_\ep\|_{L_{2, w} (\Om)} \ \le \ c\, \big(|\al|+\|
b\|_{L_{2, w} (\Om)}\big). \label{b_alpha}
\end{equation}
From Proposition  \ref{Positive_T1} we conclude that for any $\ep>0$
there exists  a $q$--weak solution  $u^\ep$ to the problem
\eqref{Equation_2D_approximate} (with $b_\ep^{(\al)}$ defined in
\eqref{b_alpha_def}).  Moreover, as the constant $C$ in
\eqref{Higher_Integrability_approx} is a non-decreasing function of
$\|b\|_{L_{2,w}(\Om)}$ from \eqref{b_alpha} we conclude that the
estimate
\begin{equation}
\| u^\ep \|_{W^1_p(\Om)} \ \le \ C\, \| f\|_{L_q(\Om)}
\label{Higher_Integrability_approx1}
\end{equation}
holds with some constant $C>0$ independent on $\ep$.
 From
\eqref{Higher_Integrability_approx1} we obtain existence of $u\in
W^1_p(\Om)$ such that for some subsequence $u^\ep$ we have
\begin{equation}
u^\ep \rightharpoonup u \quad \mbox{in} \quad W^1_p(\Om).
\label{Weak_conv}
\end{equation}
 On the other hand, from \eqref{Properties_approx}
  we conclude there exists a subsequence such that
\begin{equation}
b^{(\al)}_\ep \ \to \ b^{(\al)} \quad \mbox{in} \quad L_{p'}(\Om),
\qquad p'=\frac{p}{p-1}. \label{Strong_conv}
\end{equation}
Hence for any $\eta\in C_0^\infty(\Om)$ from \eqref{Weak_conv} and
\eqref{Strong_conv} we obtain
$$
\int\limits_\Om b^{(\al)}_\ep \cdot \nabla u^\ep \eta\, dx \ \to \
\int\limits_\Om b^{(\al)}  \cdot \nabla u \, \eta\, dx
$$
as $\ep\to 0$. Passing to the limit in the equations for $u^\ep$ (in
a weak form) we obtain \eqref{Identity_2D}. Hence $u$ is a $p$-weak
solution to the problem \eqref{Equation_2D}. The uniqueness of
$p$-weak solutions follows from \eqref{Energy}.
\end{proof}

Now we present a proof of  Theorem \ref{Theorem_2}. We split it into
several steps.

\begin{lemma}\label{4.1}
Assume  $\al<0$. Then there exists $p_1>2$ depending only on $\al$
and $\Om$ such that  for any   $B_\ep\Subset \Om$ and any  $f^\ep
\in
 C_0^\infty(\Om\setminus B_\ep)$ there exists a unique $p_1$--weak solution $w^\ep\in
 \overset{\circ}{W}{^1_{p_1}}(\Om)$  to the problem
\begin{equation}
\left\{ \quad \gathered  -\Delta w^\ep -
|\al| \, \frac{x}{|x|^2} \cdot \nabla w^\ep \ = \ -|x|^\al\div f^\ep  \qquad\mbox{in}\quad \mathcal D'(\Om), \\
w^\ep|_{\cd \Om} \ = \ 0. \qquad \qquad
\endgathered\right.
\label{b=0}
\end{equation}
\end{lemma}

\begin{proof} As the right-hand side of the equation \eqref{b=0} is smooth it can be represented in the form $\div g^\ep$ for some $g^\ep\in L_{q_1}(\Om)$ with some
$q_1>2$. Hence the result follows from
Theorem
 \ref{Theorem_1}.
\end{proof}

\begin{lemma}\label{Zero_b_negative}    Assume  $\al<0$. Then there exists $p_2>2$ depending only on $\al$ and
$\Om$ such that  for any   $B_\ep\Subset \Om$ and any  $f^\ep \in
 C_0^\infty(\Om\setminus B_\ep)$ there exists a unique $p_2$--weak solution $u^\ep$  to the problem
\begin{equation}
\left\{ \quad \gathered  -\Delta u^\ep - \al  \, \frac{x}{|x|^2} \cdot \nabla u^\ep \ = \ -\div f^\ep  \qquad\mbox{in}\quad \mathcal D'(\Om), \\
u^\ep|_{\cd \Om} \ = \ 0,  \qquad u^\ep(0) = 0. \qquad\qquad \quad
\endgathered\right.
\label{b=0_negative}
\end{equation}

\end{lemma}

\begin{proof}
Let $w^\ep$ be a unique $p_1$-weak solution to the problem
\eqref{b=0} with some $p_1>2$ depending only on $\al$ and $\Om$.
 Define the function
$$
u^\ep(x) \ := \ |x|^{|\al|} w^\ep(x).
$$
It is easy to see that  $u^\ep\in W^1_{p_2}(\Om)$ for some $p_2>2$
depending only on $p_1$ and $|\al|$. The direct computation shows
that $u^\ep$ is a $p_2$-weak solution to the problem
\eqref{b=0_negative}.
\end{proof}

\begin{lemma}\label{4.3}  Assume  $\al<0$ and $q>2$. There exists $p>2$  depending only on $q$, $\Om$ and $\al$ such that  for any
$f\in L_q( \Om)$ there exists a unique  $p$-weak solution $u$ to the
problem
\begin{equation}
\left\{ \quad \gathered  -\Delta u -\al  \, \frac{x}{|x|^2} \cdot \nabla u \ = \ -\div f  \qquad\mbox{in}\quad  \Om,   \\
u|_{\cd \Om} \ = \ 0,  \qquad u(0) = 0. \qquad\qquad \quad
\endgathered\right.
\label{b=0_negative_1}
\end{equation}
Moreover, $u$ satisfies the estimate
$$
\| u \|_{W^1_p(\Om)} \ \le \ C~\| f  \|_{L_q(\Om)}
$$
where $C>0$ is some constant depending only on $q$ and $\Om$.
\end{lemma}

\begin{proof} Assume $f\in L_q(\Om)$ and take $f^\ep\in
C_0^\infty(\Om\setminus B_\ep)$ so that $f^\ep\to f$ in $L_q(\Om)$
as $\ep\to +0$.  Denote by $p>2$ the exponent determined in Theorem
\ref{Higher_Integrability_2} (without loss of generality we can
assume $p\le p_2$ where $p_2$ is determined in Lemma
\ref{Zero_b_negative}). Denote by $u^\ep$ a unique $p$-weak solution
to the problem \eqref{b=0_negative}. Then from Theorem
\ref{Higher_Integrability_2} we obtain
$$
\| u^\ep\|_{L_p(\Om)} \ \le \ C\, \|f^\ep\|_{L_q(\Om)}
$$
and hence there exists $u\in \overset{\circ}{W}{^1_p}(\Om)$ such
that for some subsequence the weak convergence $u^\ep
\rightharpoonup u$ in $W^1_p(\Om)$ takes place. It is easy to see
that $u$ satisfies the integral identity \eqref{Identity_2D} with
$b=0$. Moreover, as $u^\ep(0)=0$ and the imbedding
$W^1_p(\Om)\hookrightarrow C(\bar \Om)$ is compact we obtain
$u(0)=0$. So, $u$ is a $p$-weak solution to the problem
\eqref{b=0_negative_1}. The uniqueness of $u$ follows from
\eqref{Energy}.
\end{proof}

 \medskip
\begin{lemma}\label{4.4}  Assume  $\al<0$,   $b\in C^\infty(\bar \Om)$, $\div b =0$ in
 $\Om$,   $q>2$ and $p>2$  is the exponent depending only on $\Om$, $q$, $\al$, $\|b\|_{L_{2,w}(\Om)}$ defined in Lemma \ref{4.3}.
 Then for any $f \in L_q(\Om)$, $v\in L_q(\Om)$  there exists a
 unique $p$-weak solution $u\in
 \overset{\circ}{W}{^1_p}(\Om)$  to the problem
 $$
\left\{ \quad \gathered  -\Delta u  -\al\, \frac{x}{|x|^2} \cdot \nabla u \ = \ -\div (f +bv) \qquad\mbox{in}\quad  \Om, \\
u|_{\cd \Om} \ = \ 0, \qquad u(0) = 0. \qquad\qquad \quad
\endgathered\right.
 $$
Moreover, the operator $A: L_q(\Om)\to L_q(\Om)$, $A(v):=u$, is
continuous and compact.
\end{lemma}

\begin{proof} For given
$f$, $v\in L_q(\Om)$ the existence of a unique $p$-weak solution $u$
satisfying $u(0)=0$  follows from Lemma \ref{4.3}. Moreover, $u$
satisfies the estimate
$$
\| u\|_{W^1_p(\Om)} \ \le \ C~(\|f\|_{L_q(\Om)} + \|
b\|_{L_\infty(\Om)} \| v\|_{L_q(\Om)}).
$$
For a given $f\in L_q(\Om)$ the last estimate implies
$$
\| A(v_1)-A(v_2)\|_{W^1_p(\Om)} \ \le \ C~ \| b\|_{L_\infty(\Om)} \|
v_1 - v_2\|_{L_q(\Om)}
$$
and hence the operator $A$ is continuous as an operator from
$L_q(\Om)$ into $L_q(\Om)$. The compactness of this operator in
$L_q(\Om)$ follows from the compactness of the imbedding of
$W^1_p(\Om)$ into $L_q(\Om)$.
\end{proof}

 \medskip
\begin{lemma}\label{4.5}
 Assume  $\al<0$,   $b\in C^\infty(\bar \Om)$, $\div b =0$ in
 $\Om$,   $q>2$ and $p>2$  is the exponent depending only on $\Om$, $q$, $\al$, $\|b\|_{L_{2,w}(\Om)}$ defined in Lemma  \ref{4.3}.
  Then for any  $f \in L_q(\Om)$  there exists
a
 unique $p$-weak solution $u $  to the problem
 $$
\left\{ \quad \gathered  -\Delta u  +\Big( b-\al \, \frac{x}{|x|^2}\Big) \cdot \nabla u \ = \ -\div f \qquad\mbox{in}\quad \Om,  \\
u|_{\cd \Om} \ = \ 0, \qquad u(0) = 0. \qquad\qquad \quad
\endgathered\right.
 $$
 \end{lemma}

\begin{proof} We apply the Leray-Schauder fixed point theorem for
the operator $A: L_q(\Om)\to L_q(\Om)$ defined in Lemma \ref{4.4}.
Assume $\la\in [0,1]$ and $v\in L_q(\Om)$ satisfies $v=\la A(v)$.
Denote $u:=A(v)$. Then $u$ is  a unique  $p$-weak solution to the
problem
$$
\left\{ \quad \gathered  -\Delta u  +\Big( \la b-\al \, \frac{x}{|x|^2}\Big)\cdot \nabla u \ = \ -\div f \qquad\mbox{in}\quad \Om, \\
u|_{\cd \Om} \ = \ 0, \qquad u(0) = 0. \qquad\qquad \quad
\endgathered\right.
 $$
 From Theorem \ref{Higher_Integrability_2} we obtain the estimate
 $$
 \| u\|_{W^1_p(\Om)} \ \le \ C\, \| f\|_{L_q(\Om)}
 $$
 with some constant depending only on $q$, $\Om$, $\al$ and $\|
 b\|_{L_{2,w}(\Om)}$ and independent of $\la\in [0,1]$. Hence there
 exists   $u\in \overset{\circ}{W}{^1_p}(\Om)$
 satisfying $u=A(u)$.
\end{proof}

Finally, we can relax the smoothness conditions on the
divergence-free part of the drift and prove Theorem \ref{Theorem_2}:

\begin{proof} Assume $b$ satisfies \eqref{Assumptions_b_2D} and let $p>2$ be the exponent defined in  Lemma \ref{4.5}. From
Lemma \ref{Approximation_result} we obtain existence of $b_\ep\in
C^\infty(\bar \Om)$ such that $\div b_\ep=0$ in $\Om$, $\|
b_\ep\|_{L_{2,w}(\Om)}\le c\, \| b\|_{L_{2,w}(\Om)}\le $  and $b_\ep
\to b$  in $L_{p'}(\Om)$ where $p'=\frac{p}{p-1}$. From Lemma
\ref{4.5} we conclude that for any $\ep>0$ there exists a unique
$p$-weak solution $u^\ep$ to the problem
$$
\left\{ \quad \gathered  -\Delta u^\ep  +b^{(\al)}_\ep \cdot \nabla u^\ep \ = \ -\div f \qquad\mbox{in}\quad \Om, \\
u^\ep|_{\cd \Om} \ = \ 0, \qquad u^\ep(0) = 0, \qquad\qquad \quad
\endgathered\right.
 $$
 where $b^{(\al)}_\ep = b_\ep -\al \frac{x}{|x|^2}$. From Theorem
 \ref{Higher_Integrability_2} we obtain the estimate
 $$
\| u^\ep \|_{W^1_p(\Om)} \ \le \ C\, \|f\|_{L_q(\Om)}
 $$
 where the constant $C>0$ depends only on $a$, $\Om$, $\al$ and $\|
 b\|_{L_{2,w}(\Om)}$. Then there exists  $u\in \overset{\circ}{W}{^1_p}(\Om)$  such that  for some subsequence
 $u^\ep$ the convergence \eqref{Weak_conv} holds. The rest of the
 proof repeats the proof of Theorem \ref{Theorem_1} in the beginning of this section. The identity
 $u(0)=0$ follows from $u^\ep(0)=0$ and compactness of the imbedding
 of $W^1_p(\Om)$ into $C(\bar \Om)$. The uniqueness of $p$-weak
 solutions follows from \eqref{Energy}.
\end{proof}

\newpage
\section{Appendix}

\bigskip

Here we present a variant of the imbedding theorem  that we used in
Section \ref{A priori estimates Section}:

\begin{proposition} \label{Poincare_inequality}
Assume $B_R\subset \mathbb R^n$, $n\ge 2$,  $p\in [1, n)$ and $p_*
:=\frac{pn}{n-p}$. Then for any $\la \in (0,1)$ there exists
$C=C(n,p )>0 $ such that if $u\in W^1_p(B_R)$ satisfies  $$|\{ \,
x\in B_R: \, u(x)=0 \, \}| \ \ge \ \la |B_R|$$ then
$$
\| u\|_{L_{p_*}(B_R)} \ \le \ C(n,p)\, \la^{-1}~\|\nabla
u\|_{L_p(B_R)}.
$$

\end{proposition}

\begin{proof} From the Sobolev imbedding
theorem we conclude
$$
\| u\|_{L_{p_*}(B_R)} \ \le \ \frac{c}R\,\| u\|_{L_p(B_R)}
+c~\|\nabla u\|_{L_p(B_R)}.
$$
Denote $E:=\{ \, x\in B_R: \, u(x)=0 \, \}$ and $u_\Om:= \pint_\Om
u~dx$. Then we have
$$
\gathered  \| u\|_{L_p(B_R)} \ = \ \| u-u_E\|_{L_p(B_R)} \ \le \ \|
u-u_{B_R}\|_{L_p(B_R)} +|B_R|^{   1/p}|u_E-u_{B_R}| \ \le \\ \le \
\| u-u_{B_R}\|_{L_p(B_R)} \ +\ \frac{|B_R|^{
1/p}}{|E|}\|u-u_{B_R}\|_{L_1(B_R)}.
\endgathered
$$
Now the result follows from the usual Poincare and H\" older
inequalities:
$$
\gathered \| u-u_{B_R}\|_{L_p(B_R)} \ \le \ c R\, \|\nabla
u\|_{L_p(B_R)}
\\
\|u-u_{B_R}\|_{L_1(B_R)} \ \le \ |B_R|^{1/p'}
\|u-u_{B_R}\|_{L_p(B_R)} \ \le \ c\, |B_R|^{1/n+1/p'} \|\nabla u
\|_{L_p(B_R)}.
\endgathered
$$
\end{proof}

\end{document}